\documentclass{amsart}
\usepackage{amsmath,amsfonts,amssymb,amsthm}
\usepackage{epsfig}
\usepackage{mathrsfs} 
\usepackage{enumerate}
\newtheorem{theorem}{Theorem}

\newtheorem{koro}[theorem]{Corollary}
\newtheorem{thm}{Theorem}[section]

\newtheorem{prop}[thm]{Proposition}

\newcommand{\lin}{{\mathrm{lin}}\,}

\newcommand{\R}{{\mathbb R}}

\catcode`@=11
\def\section{%
\setcounter{equation}{0} \setcounter{theorem}{0} \@startsection
{section}{1}{\z@}{-4.0ex plus -1ex minus
    -.2ex}{2.3ex plus .2ex}{\bf}}
\catcode`@=12 \theoremstyle{definition}

\begin{document}

\title{Grassmann measures of convex bodies}

\author{Wolfgang Weil}
\address{Karls\-ruhe Institute of Technology, Department of Mathematics,
D-76128 Karls\-ruhe, Germany}
\email{wolfgang.weil@kit.edu}
\urladdr{http://www.math.kit.edu/$\sim$weil/}
\date{\today}

\begin{abstract}
\noindent Flag measures are descriptors of convex bodies $K$ in $d$-dimensional Euclidean space generalizing the classical area measures. They have been used to provide general integral formulas for mixed volumes (see Hug, Rataj and Weil (2017)). Here, we consider an image measure $\gamma_j(K,\cdot)$ of flag measures, defined on the Grassmannian $G(d,j)$ of affine $j$-spaces, $1\le j\le d-1$, and show that it determines centrally symmetric bodies $K$ of dimension $\geq j+1$ uniquely. We then explain that Grassmann measures appear in the representation of smooth, translation invariant, continuous and even valuations due to Alesker (2003). Using this connection, we prove a uniqueness result for projection averages of area measures and we finally discuss a Grassmann version of the natural touching measure of convex bodies. 
\end{abstract}
\maketitle

\section{Introduction}

Let ${\mathcal K}^d$ be the space of  convex bodies (non-empty compact convex sets) in $\R^d, d\ge 2,$ supplied with the Hausdorff metric. Classical descriptors of bodies $K\in{\mathcal K}^d$ are the area measures $\Psi_j(K,\cdot)$, $j=1,\dots ,d-1$. These are finite Borel measures on the unit sphere $S^{d-1}$, which arise from a local Steiner formula and which determine the body $K$ uniquely, up to a translation (and under a dimensional condition). The total measure $V_j(K)=\Psi_j(K,S^{d-1})$ is the $j$th intrinsic volume of $K$. See \cite{S}, for definitions, properties and further results on area measures and intrinsic volumes, as well as for all other notions from convex geometry which appear throughout the following. 

As a more recent development, measures on flag manifolds have been introduced and studied to describe convex bodies, a survey is given in \cite{HTW}. They also have a history in geometric measure theory (see \cite{HTW}, for details). Flag measures allow to extend formulas for mixed volumes and projection functions from special cases, where area measures are involved, to more general situations (see \cite{GHHRW, HRW, HRW2}). There are several isomorphic representations possible for these measures. In the following, we concentrate on the measure $\psi_j(K,\cdot)$, which is a finite measure on the flag space 
$$F(d,j+1):=\{(u,L) : u\in L\cap S^{d-1}, L\in G(d,j+1)\}, \ j=1,\dots ,d-1.$$ 
Here, $G(d,k)$ denotes the Grassmannian of $k$-dimensional subspaces in $\R^d$. Since the area measure $\Psi_j(K,\cdot)$ appears as the image of $\psi_j(K,\cdot)$ under the projection $(u,L)\mapsto u$, the flag measure contains more information on the boundary structure of a body $K$, in case $j\le d-2$. Clearly, $\psi_j(K,\cdot)$ also determines the body $K$ (up to translations and under the restrictions mentioned above). It is a first goal of this paper to  study another image of flag measures, namely the measure $\gamma_j(K,\cdot)$, $j=1,...,d-1$, on $G(d,j)$, which arises from the mapping $(u,L)\mapsto u^\bot\cap L$. It may be called the $j$th {\it Grassmann measure} of $K$. For $j=d-1$, this image measure just corresponds to the symmetrized area measure $\frac{1}{2}[\Psi_{d-1}(K,\cdot)+\Psi_{d-1}(-K,\cdot)]$, if we identify the subspace $u^\bot\cap\R^d$ with the vector $u\in S^{d-1}$ (or with $-u$). Therefore, we concentrate on the case $j\le d-2$, in the following. Since $\gamma_j(K,\cdot)$ is invariant under translations $K\mapsto K+x, x\in\R^d,$ and reflections $K\mapsto -K$, a uniqueness result can only be expected for $K$ in the class ${\mathcal K}^d_s =\{ K\in{\mathcal K}^d : K=-K\}$ of centrally symmetric convex bodies. The following is  our first result. It involves the invariant probability measure $\nu_j$ on $G(d,j)$ and the projection function $L\mapsto V_j(K|L), L\in G(d,j),$ of a convex body $K$.

\begin{theorem}\label{grassmann} For $j\in\{ 1,\dots ,d-2\}$, we have 
$$\gamma_j(K,A) = \int_{A} V_j(K|L)\nu_j(dL)$$
for all $K\in{\mathcal K}^d$ and all Borel sets $A\subset G(d,j)$.
\end{theorem}

\begin{koro}\label{unique-1} Let $K,K'\in{\mathcal K}^d_s$ with $\dim K>j$, $\dim K'>j$, $j\in\{ 1,\dots ,d-2\}$. Then, $\gamma_j(K,\cdot )=\gamma_j(K',\cdot)$ implies $K=K'$.
\end{koro}

In the proof of Theorem \ref{grassmann}, we will use the connection between Grassmann measures and even valuations. We recall, that
a real-valued functional $\varphi$ on ${\mathcal K}^d$ is called a {\it valuation}, if it is additive in the sense that
$$
\varphi (K\cup M) + \varphi (K\cap M) = \varphi (K) + \varphi (M) ,
$$
whenever $K,M$ and $K\cup M$ lie in ${\mathcal K}^d$. In the following, all valuations are assumed to be translation invariant and continuous (w.r.t. the Hausdorff metric). Let $\bf Val$ be the vector space of all these valuations and let ${\bf Val}^+$ (resp. ${\bf Val}^-$) be the subspace of {\it even} (resp. {\it odd}) valuations. For $j\in\{0,...,d\}$, a  valuation $\varphi\in{\bf Val}$ is {\it $j$-homogeneous}, if
$$
\varphi (\alpha K) = \alpha^j \varphi(K),\quad K\in{\mathcal K}^d, \alpha\ge 0.$$
McMullen \cite{McM77,McM93} has shown that there is a direct decomposition
\begin{equation}\label{val}
 {\bf Val}= \bigoplus_{j=0}^{d} {\bf Val}_j= \bigoplus_{j=0}^{d} ({\bf Val}_j^+\oplus {\bf Val}_j^-)
\end{equation}
into the spaces of (even and odd) $j$-homogeneous valuations (which are again translation invariant and continuous). Here, the cases $j=0$ and $j=d$ are simple (the valuations are constants, respectively constant multiples of the volume $V_d$), therefore we concentrate on $j\in\{1,\dots,d-1\}$, in the following.   For $j=d-1$, McMullen \cite{McM80} has shown that $\varphi\in {{\bf Val}}_{d-1}$, if and only if
\begin{equation}\label{(d-1)case}
\varphi (K) = \int_{S^{d-1}} f(u) \Psi_{d-1} (K,du),\quad K\in{\mathcal K}^d,
\end{equation}
for some continuous function $f= f_\varphi$ on the unit sphere $S^{d-1}$ which is uniquely determined, up to a linear function.  It follows that $\varphi$ is even, if $f$ is even, and odd, if $f$ is odd. For $1\le j\le d-2$, a similar description of ${{\bf Val}}_j$ is not possible, at least not with area measures. For the subclass ${{\bf Val}}_j^{\infty,+} \subset {{\bf Val}}_j^+$ of smooth valuations we have the following result.

\begin{theorem}\label{charact-1} For a valuation $\varphi\in {\bf Val}_j^{\infty,+}$, $j\in\{1,\dots,d-2\}$, there is a $C^\infty$-function $f$ on $G(d,j)$ such that
\begin{align*}
\varphi (K) &= \int_{F(d,j)} f(u^\perp\cap L) \psi_j(K,d(u,L))\\
&=\int_{G(d,j)} f(L) \gamma_j(K,dL) = \int_{G(d,j)} f(L) V_j(K|L)\nu_j(dL)
\end{align*}
for all $K\in{\mathcal K}^d$.
\end{theorem}

For $\varphi\in {\bf Val}_j^{\infty,+}$, the representation
\begin{equation}\label{alesker}
\varphi (K) = \int_{G(d,j)} f(L) V_j(K|L)\nu_j(dL)
\end{equation}
with a smooth function $f$ follows from Alesker's irreducibility theorem and is contained in Alesker's work (see \cite[p. 73]{A1}). Here, we obtain the result through the connection with Grassmann measures.
We remark that, for $j=d-1$, the theorem also holds (with the above-mentioned interpretation of $\gamma_{d-1}(K,\cdot)$), but reduces to (the symmetric and smooth version of) McMullen's result \eqref{(d-1)case}. For $1<j< d-1$, the function $f$ is no longer unique, as we shall also see.

In the course of explaining the connection between Theorem \ref{grassmann} and Theorem \ref{charact-1}, we come across a third topic which we want to discuss, namely projection averages. As a generalization of the classical Cauchy-Kubota formula for intrinsic volumes (see \cite{S} or \cite{SW}), Schneider \cite{Sch75} proved a projection formula for area measures,
\begin{equation}\label{projaver}
\int_{G(d,k)}\Psi'_j(K|E,A\cap E)\nu_{k}(dE) = c_{d,k-j}^{k,d-j} \Psi_j(K,A)
\end{equation}
for $1\le j<k\le d-1$, $K\in{\mathcal K}^d$ with $\dim K\ge k$ and all Borel sets $A\subset S^{d-1}$. Here, $\Psi'_j(K|E,\cdot)$ is the area measure of $K|E$ calculated in the subspace $E$ (and hence a measure on $S^{d-1}\cap E$) and $c_{d,k-j}^{k,d-j}$ is a dimensional constant (see Section 2). Of course, this implies that the projection average
$$
\int_{G(d,k)}\Psi'_j(K|E,\cdot\cap E)\nu_{k}(dE)
$$
determines $K$ uniquely (up to translations and under the mentioned dimensional restriction). If we consider the corresponding integral
$$
\int_{G(d,k)}\Psi_j(K|E,\cdot)\nu_{k}(dE) ,
$$
where the $j$th area measure of the projection $K|E$ is averaged (as a measure on $S^{d-1}$), the resulting measure  is homogeneous of degree $j$ in $K$, but is not proportional to the area measure $\Psi_j(K,\cdot)$ anymore (in fact, it is not locally defined). Thus the question arises whether this projection average also determines $K$. We obtain a positive result for centrally symmetric bodies.

\begin{theorem}\label{projaver-1} Let $K,K'\in{\mathcal K}^d_s$ with $\dim K>j$, $\dim K'>j$, $j\in\{ 1,\dots ,d-2\}$. Then, 
$$
\int_{G(d,k)}\Psi_j(K|E,\cdot)\nu_{k}(dE) = \int_{G(d,k)}\Psi_j(K'|E,\cdot)\nu_{k}(dE) ,
$$ 
for $j<k\le d-1$, implies $K=K'$.
\end{theorem}

Surprisingly, the corresponding result is wrong, in general, if the symmetry of the bodies $K,K'$ is skipped. We shall give a corresponding example in the case $j=1$.

Finally, we discuss the connection between the Grassmann measures and touching measures, as they were treated in \cite[Section 8.6]{SW}. As a further result, we obtain that a centrally symmetric body $K$ of dimension $> j$ is uniquely determined (up to translation) by its $(d-j-1)$st {\it direction measure} $\nu_{d-j-1}(K,\cdot)$, $j\in\{ 1,\dots, d-2\}$. The latter is the image of the touching measure $\mu_{d-j-1}(K,\cdot)$ defined in \cite[p. 358]{SW}, if the affine $(d-j-1)$-dimensional subspaces $E$ touching $K$ are mapped to their direction spaces $L(E)\in G(d,d-j-1)$.

\begin{theorem}\label{unique-2} 
 Let $K,K'\in{\mathcal K}^d_s$ with $\dim K>j$, $\dim K'>j$, $j\in\{ 1,\dots ,d-2\}$. Then, $\nu_{d-j-1}(K,\cdot )=\nu_{d-j-1}(K',\cdot)$ implies $K=K'$.
\end{theorem}

The setup of the paper is as follows. 
In the next section, we collect some further basic facts from convex geometry which we will use, including a description of various flag measures. We then discuss the Grassmann measures and collect some of their simple properties. In Section 3, we collect more information about valuations which is used later on. In the subsequent sections we give the proofs of Theorems \ref{grassmann}, \ref{charact-1}, \ref{projaver-1} and \ref{unique-2}.

\section{Facts from convex geometry}

Measures on flag manifolds as descriptors of convex bodies are of more recent origin. They can be introduced by a local Steiner formula in the affine Grassmannian $A(d,k)$ which defines a whole series of flag measures or by projection averages of area measures. There is also a direct approach which uses measure geometric tools and defines flag measures as integrals with respect to the Hausdorff measure on the generalized normal bundle of the body. We refer to \cite{HTW}, for a survey on flag measures which contains also historical remarks and gives further references. In the following, we concentrate on the measure $\psi_j(K,\cdot )$, for $K\in{\mathcal K}^d$ and $j\in\{ 1,\dots, d-1\}$, which we define as a measure on the flag space $F(d,j+1)$ by a projection mean. Here, for $1\le k\le d$, 
$$
F(d,k) = \{ (u,E) : E\in G(d,k), u\in S^{d-1}\cap E\} .
$$
Then,
\begin{equation}\label{flagprojformula}
\psi_j (K,\cdot) = \int_{G(d,j+1)}\int_{S^{d-1}\cap E} {\bf 1}\{ (u,E)\in\cdot\} \Psi'_j(K|E,du) \nu_{j+1}(dE),
\end{equation}
where $\Psi_{j}' (K|E,\cdot)$ denotes the area measure of $K|E$ computed in $E$ and $\nu_{j+1}$ is the rotation invariant probability measure on $G(d,j+1)$.  
For the necessary measurability properties needed here and in the following, we refer to \cite{Hind}. As a variant of $F(d,k)$ one can use
$$
F^\bot(d,d-k) = \{ (u,E) : E\in G(d,d-k), u\in S^{d-1}\cap E^\bot\} ,
$$
which is obviously homeomorphic to $F(d,k)$ under the orthogonality map $(u,L)$ $\mapsto (u,L^\perp)$ (both spaces, $F(d,k)$ and $F^\bot(d,d-k)$ carry natural topologies). Thus, $\psi_j (K,\cdot)$  can also be interpreted as a measure on $F^\bot(d,d-j-1)$. Another isomorphic representation results, if we consider the image of $\psi_j (K,\cdot)$ under the bijection $(u,E)\mapsto (u,E^\bot\vee u)$, where $E^\bot\vee u$ denotes the linear hull of $E^\bot\cup\{u\}$. The resulting measure sits on $F(d,d-j)$ and is the measure considered in \cite[eq. (3.1)]{HHW} (first defined for polytopes by a sum over all $j$-faces and extended to arbitrary $K\in{\mathcal K}^d$ by continuity). This latter measure was also used in \cite{GHHRW}.  In \cite{HTW}, as well as in \cite{Hind}, a series of flag measures $S^{(k)}_m(K,\cdot)$ on $F^\perp(d,k)$ was introduced by a local Steiner formula (see also \cite[p. 5]{GHHRW}). Here, the connection to the measure $\psi_j(K,\cdot)$ is given by
\begin{align}\label{connect}                                                  
\int_{F(d,j+1)}&f(u,L)\psi_j(K,d(u,L))  \cr 
&= c_{d,j} \int_{F^\perp(d,d-j-1)} f(v,M^\bot)S^{(d-j-1)}_{j}(K,d(v,M)) 
\end{align}
(this is essentially formula (2.1) in \cite{GHHRW}, where also the constant $c_{d,j}$ is given). 
In \cite{HRW}, a flag measure $\Omega_j(K;\cdot)$ on $F^\perp(d,d-1-j)$ was used in a measure geometric context, which satisfies
$$
\Omega_j(K;\cdot) = \tilde c_{d,j} S^{(d-1-j)}_j(K,\cdot)
$$
with a constant $\tilde c_{d,j}$ given explicitly in \cite[eq. (2.4)]{GHHRW}.

The image measure of $\psi_j(K,\cdot )$ under the mapping $(u,L)\mapsto u$ is proportional to the area measure $\Psi_j(K,\cdot )$, as follows from the local projection formula \eqref{projaver},
\begin{equation*}
\int_{G(d,j+1)}\Psi_j'(K|L,L\cap\cdot)\nu_{j+1}(dL) = c^{j+1,d-j}_{d,1} \Psi_j(K,\cdot)
\end{equation*}
for $j=1,\dots, d-2$. Here
$$
c^{k,m}_{l,n}= \frac{k!\kappa_k m!\kappa_m}{l!\kappa_l n!\kappa_n},$$
where $\kappa_i$ is the ($i$-dimensional) volume of the unit ball in $\mathbb R^i$. 
This shows, that the flag measure $\psi_j(K,\cdot )$ determines bodies $K$ of dimension $\ge j+1$ uniquely (up to a translation). Since $\psi_{d-1} (K,\cdot)=\Psi_{d-1}(K,\cdot)$ (if we identify $(u,\R^d)$ and $u$), we are mostly interested in the case $j\le d-2$, in the following.

For $1\le j\le d-2$, we now define the $j$th {\it Grassmann measure} $\gamma_j(K,\cdot)$ on $G(d,j)$ as the image measure of $\psi_j(K,\cdot )$ under the mapping $(u,L)\mapsto u^\bot\cap L$. The Grassmann measures inherit some properties from the corresponding flag measures. Namely,  $\gamma_j(K,\cdot)$ is translation invariant and rotation covariant (that is, $\gamma_j(\vartheta K,A) = \gamma_j (K,\vartheta^{-1}A)$ for $K\in{\mathcal K}^d, A\subset G(d,j)$ a Borel set and $\vartheta\in SO(d)$). The mapping $K\mapsto \gamma_j(K,\cdot)$ is additive (hence a valuation), homogeneous of degree $j$ and continuous (w.r.t. the weak* topology for measures and the Hausdorff metric on ${\mathcal K}^d$). 
There is also a (local) Steiner-type formula for Grassmann measures. In fact, $\gamma_j(K+tB^d,\cdot)$ is a polynomial in $t\ge 0$ and the coefficients are measures on $G(d,j)$ indexed by two parameters. If $K$ has dimension $<j$, then $\gamma_j(K,\cdot)=0$. If $\dim K=j$, then $\gamma_j(K,\cdot) =V_j(K)\delta_{L(K)}$, where $V_j(K)$ denotes the $j$th intrinsic volume of $K$ and $\delta_{L(K)}$ is the Dirac measure concentrated on the linear space $L(K)\in G(d,j)$ parallel to the affine hull of $K$ .

For $K\in{\mathcal K}^d_s$ and a Borel set $A\subset G(d,j)$, \eqref{flagprojformula} implies
\begin{align}\label{Grassmeas}
\gamma_j(K,A) &= \int_{G(d,j+1)}\int_{S^{d-1}\cap E} {\bf 1}\{ u^\perp\cap E\in A\} \Psi'_j(K|E,du) \nu_{j+1}(dE)\nonumber\\
&= \int_{G(d,j+1)} \Psi'_j(K|E,A_E) \nu_{j+1}(dE) \nonumber\\
&= \int_{G(d,j+1)}\tilde \Psi_{j,E}(K|E,A\cap E) \nu_{j+1}(dE),
\end{align}
where
$$
A_E = \{ u\in E\cap S^{d-1} : u^\perp\cap E\in A\} ,
$$
$$
A\cap E = \{ L\in G(d,j) : L\in A, L\subset E\} ,
$$
and where $\tilde\Psi_{j,E}(K|E,\cdot)$ is the image measure on $G(d,j)$ of $\Psi'_j(K|E,\cdot)$ under the mapping $u\mapsto u^\perp\cap E$. The representation \eqref{Grassmeas}  resembles the projection formula \eqref{projaver}, but the latter is not directly applicable here. 

It might appear natural to consider, as an alternative to $\gamma_j(K,\cdot)$, the image $\tilde\gamma_j(K,\cdot)$ of $\psi_j(K,\cdot)$ under the projection $(u,L)\mapsto L$. This image is a measure on $G(d,j+1)$. As we shall see in Section 7, it is directly related to the touching measure $\mu_{d-j-1}(K,\cdot)$ and is connected to $\gamma_j(K,\cdot)$ by a Radon transform. For the applications to even valuations and to projection averages, the measure $\gamma_j(K,\cdot)$ seems to be more appropriate.

There is another measure on the Grassmannian $G(d,j)$ which describes centrally symmetric bodies $K$ from a dense subclass of ${\mathcal K}^d_s$ (the generalized zonoids). This is the $j$th {\it projection generating measure} $\rho_{(j)}(K,\cdot)$ which was first mentioned (with this name)  in \cite[Section 6]{GW93}, but appeared in equivalent form in various earlier formulas for zonoids and generalized zonoids (see also \cite[p. 308]{S}) . The latter is a finite (signed) measure on $G(d,j)$ defined as the image of the measure
$$
\frac{2^j}{j!} \int_{(\cdot)}\cdots\int_{(\cdot)}D_j(u_1,\dots ,u_j)\rho(K,du_1)\cdots\rho(K,du_j) 
$$
on $(S^{d-1})^j$ under the mapping $(u_1,\dots, u_j)\mapsto \lin \{ u_1,\dots ,u_j\}$. Here, $\rho(K,\cdot)$ is the {\it generating measure} of $K$, $\lin\{u_1,\dots ,u_j\}$ is the linear subspace generated by $u_1,\dots ,u_j$ and $D_j(u_1,\dots ,u_j)$ is the absolute value of the determinant of $u_1,\dots ,u_j$ in this subspace (this value is interpreted as $0$, if $u_1,\dots ,u_j$ are not linearly independent). The name for $\rho_{(j)}(K,\cdot)$ comes from the fact that
\begin{equation}\label{projgenmeas}
V_j(K|L) = \int_{G(d,j)} \langle L,M\rangle \rho_{(j)}(K,dM),\quad L\in G(d,j),
\end{equation}
for a generalized zonoid $K$ (see \cite[eq. (6.1)]{GW93}). Here, $\langle L,M\rangle$ is the absolute value of the determinant of the orthogonal projection of $L$ onto $M$. The measure $\rho_{(d-1)}(K,\cdot)$ is (up to the identification $u^\perp
\longleftrightarrow \{u,-u\}$) proportional to the area measure $\Psi_{d-1}(K,\cdot)$, therefore \eqref{projgenmeas} holds in this case $j=d-1$ for all $K\in{\mathcal K}^d_s$. For $j \le d-2$ however, there is no continuous extension of \eqref{projgenmeas} to all bodies $K\in{\mathcal K}^d_s$ with a measure. One has to use distributions on $G(d,j)$.

\section{More on valuations}

Besides general valuations we also consider smooth ones. Here a valuation $\varphi\in {\bf Val}$ is called {\it smooth}, if the mapping $\phi : GL(d)\to {\bf Val}, A\mapsto A\varphi$, is infinitely differentiable (see \cite[Sec.~6.5]{S}). The subspaces ${{\bf Val}}_j^{\infty}$, ${{\bf Val}}_j^{\infty,+}$, ${{\bf Val}}_j^{\infty,-}$ of smooth valuations in ${{\bf Val}}_j$, ${{\bf Val}}_j^+$, ${{\bf Val}}_j^-$ are dense. 

For $\varphi\in {{\bf Val}}_j$, the {\it flag function} ${\bf I}_j\varphi$ of $\varphi$ is the (uniquely determined) centred continuous function on $F(d,j+1)$, defined by 
\begin{equation}\label{flagfunction}
\varphi (K|L) = \int_{S^{d-1}\cap L} {\bf I}_j\varphi(u,L) \Psi_{j}' (K|L,du)
\end{equation}
for $(u,L)\in F(d,j+1)$. Here, we have used McMullen's result \eqref{(d-1)case} in $L$.  The mapping 
${\bf I}_j : \varphi\mapsto {\bf I}_j \varphi$ from ${\bf Val}_j$ into the Banach space $C_o(F(d,j+1))$ of centred continuous functions on $F(d,j+1)$ is linear, continuous and injective (see, e.g., \cite{GHW2}). The injectivity  implies that $\varphi\in{\bf Val}_j$ is even (odd), if and only if ${\bf I}_j\varphi$ is even (odd). 
The restriction of ${\bf I}_j$ to even valuations corresponds to the {\it Klain embedding} and shows that 
\begin{equation}\label{Klain}
({\bf I}_j\varphi) (u,L) + ({\bf I}_j\varphi) (-u,L) = 2 ({\bf K}_j \varphi)(L\cap u^\bot),
\end{equation}
 where ${\bf K}_j \varphi$ is the Klain function of $\varphi$, a continuous function on $G(d,j)$. 

The following result gives a connection between the Grassmann measure and the Klain function.

\begin{theorem} \label{Grassmannmeas}
Let $K\in {\mathcal K}^d_s$ and $\varphi \in{\bf Val}_j^+$, $j\in\{ 1,\dots, d-2\}$. Then 
\begin{equation}\label{Grassmann-1}
\int_{G(d,j)} ({\bf K}_j\varphi)(M)\gamma_j(K,dM) = \int_{G(d,j+1)} \varphi (K|L)\nu_{j+1}(dL) .
\end{equation}
\end{theorem}

\begin{proof} By definition of the Grassmann measure and equations \eqref{Klain}, \eqref{flagprojformula} and \eqref{flagfunction}, we have
\begin{align*}
\int_{G(d,j)} ({\bf K}_j\varphi)(M)&\gamma_j(K,dM)\\
&= \int_{F(d,j+1)} ({\bf K}_j\varphi)(u^\bot\cap L)\psi_j(K,d(u,L))\\
&= \int_{F(d,j+1)} ({\bf I}_j\varphi)(u,L)\psi_j(K,d(u,L))\\
&= \int_{G(d,j+1)} \int_{S^{d-1}\cap E}({\bf I}_j\varphi)(u,E)\Psi_j'(K|E,du)\nu_{j+1}(dE)\\
&= \int_{G(d,j+1)} \varphi(K|E)\nu_{j+1}(dE).
\end{align*}
\end{proof}

If a valuation $\varphi\in {\bf Val}_j^+$ satisfies
$$
\varphi (K) = \int_{G(d,j)} V_j(K|L)\mu (dL), \quad K\in{\mathcal K}^d_s ,
$$
with a finite (signed) Borel measure $\mu$ on $G(d,j)$, this measure is called a {\it Crofton measure} for $\varphi$ (see \cite{Bernig, Schuster}). The name arises from the fact that $\varphi$ is proportional to the Crofton-type integral
$$
\int_{A(d,d-j)} V_0(K\cap E) \tilde\mu (dE),
$$
where $\tilde\mu$ is the translation invariant measure on the affine Grassmannian $A(d,d-j)$ with direction measure $\mu^\perp$, the image measure of $\mu$ under $L\mapsto L^\perp$ (see \cite{SW}). Crofton measures play an important role in the geometry of Finsler spaces and also in the description of even Minkowski valuations $\Phi : {\mathcal K}^d_s\to {\mathcal K}^d_s$. Suppose $\varphi\in {\bf Val}_j^+$ admits a Crofton measures $\mu$. Then, for a $j$-dimensional body $K$ in a subspace $M\in G(d,j)$,
\begin{align*}
\varphi (K) &= ({\bf K}_j\varphi)(M) V_j(K)\\
&= \int_{G(d,j)} V_j(K|L)\mu (dL)\\
&=\int_{G(d,j)}  \langle L,M\rangle V_j(K)\mu (dL) ,
\end{align*}
hence
\begin{equation}\label{croftonmeas}
{\bf K}_j\varphi = {\bf C}_j \mu .
\end{equation}
In the opposite direction, if \eqref{croftonmeas} holds for some $\varphi\in {\bf Val}_j^+$ and a measure $\mu$ on $G(d,j)$, then $\mu$ is a Crofton measure for $\varphi$. This follows easily, if a valuation $\varphi'$ is defined by 
$$
\varphi' (K) = \int_{G(d,j)} V_j(K|L)\mu (dL)
$$
for $K\in{\mathcal K}^d.$ Then $\varphi'\in {\bf Val}_j^+$ and has Crofton measure $\mu$. Therefore, as we have just seen,
$$
{\bf K}_j\varphi' = {\bf C}_j \mu = {\bf K}_j\varphi,
$$
which implies $\varphi' = \varphi$.

\section{Proof of Theorem \ref{grassmann}}

The proof is based on the connection between the Grassmann measure $\gamma_j(K,\cdot)$ and the projection generating measure $\rho_{(j)}(K,\cdot)$. 
In \cite[Corollary 3 and Remark 4]{HTW} the flag measure $\psi_j(K,\cdot)$ of a generalized zonoid $K$ was expressed in terms of the generating measure $\rho(K,\cdot)$. Using this result (and the notions explained there), we get
\begin{align*}
\psi_j(K,\cdot)&= \frac{2^j}{j!} \int_{S^{d-1}}\cdots\int_{S^{d-1}}\int_{G(d,j+1)} D_j(u_1|L,\dots ,u_j|L)\\
&\quad\times{\bf 1}\{(T^L(\pi_L(u_1),\dots ,\pi_L(u_j)),L)\in\cdot\}\nu_{j+1}(dL) \rho(K,du_1)\cdots\rho(K,du_j)\\
&= \frac{2^j}{j!} \int_{G(d,j+1)}\int_{S^{d-1}}\cdots\int_{S^{d-1}} D_j(u_1|L,\dots ,u_j|L)\\
&\quad\times{\bf 1}\{(T^L(\pi_L(u_1),\dots ,\pi_L(u_j)),L)\in\cdot\} \rho(K,du_1)\cdots\rho(K,du_j)\nu_{j+1}(dL)\\
&= \frac{2^j}{j!} \int_{G(d,j+1)}\int_{S^{d-1}}\cdots\int_{S^{d-1}} D_j(u_1,\dots ,u_j)\langle \lin (u_1,\dots,u_j),L\rangle\\
&\quad\times{\bf 1}\{(T^L(\pi_L(u_1),\dots ,\pi_L(u_j)),L)\in\cdot\} \rho(K,du_1)\cdots\rho(K,du_j)\nu_{j+1}(dL),
\end{align*}
hence
\begin{align*}
\gamma_j(K,\cdot)&= \frac{2^j}{j!} \int_{G(d,j+1)}\int_{S^{d-1}}\cdots\int_{S^{d-1}} D_j(u_1,\dots ,u_j)\langle \lin (u_1,\dots,u_j),L\rangle\\
&\quad\times{\bf 1}\{\lin (u_1,\dots,u_j)|L\in\cdot\} \rho(K,du_1)\cdots\rho(K,du_j)\nu_{j+1}(dL)\\
&=\int_{G(d,j+1)}\int_{G(d,j)}{\bf 1}\{M|L\in\cdot\}\langle M,L\rangle \rho_{(j)}(K,dM)\nu_{j+1}(dL).
\end{align*}
Since $\langle M,L\rangle=\langle M,M|L\rangle$ and using Fubini and the fact that
$$
\int_{G(d,j+1)}{\bf 1}\{M|L\in\cdot\}\langle M,M|L\rangle \nu_{j+1}(dL) = \int_{G(d,j)}{\bf 1}\{N\in\cdot\}\langle M,N\rangle \nu_{j}(dN),$$
for fixed $M\in G(d,j)$ (here the condition $j\le d-2$ is essential), we arrive at
\begin{equation}\label{gm-pgm}
\gamma_j(K,\cdot) = \int_{G(d,j)}\left(\int_{G(d,j)} \langle M,N\rangle\rho_{(j)}(K,dM)\right){\bf 1}\{N\in\cdot\}\nu_{j}(dN).
\end{equation}

The inner integral is the cosine transform $[{\bf C}_j\rho_{(j)}(K,\cdot)](N)$ of $\rho_{(j)}(K,\cdot)$ on the Grassmannian $G(d,j)$. From Theorem 5.3.5 in \cite{S} and the subsequent remarks (on p. 306), we get that
$$
V_j(K|L) = [{\bf C}_j\rho_{(j)}(K,\cdot)](L) ,
$$
hence Theorem \ref{grassmann} is proved for generalized zonoids $K$. By approximation, it follows for arbitrary $K\in {\mathcal K}^d_s$.

In order to extend the result to bodies $K\in{\mathcal K}^d$, we define two valuations $\varphi_1, \varphi_2$ by
$$
\varphi_1(K) = \int_{G(d,j)} g(L) \gamma_j(K,dL)
$$
and
$$
\varphi_2(K) = \int_{G(d,j)} g(L) V_j(K|L) \nu_j(dL)
 $$
for $K\in {\mathcal K}^d$. Here $g$ is an arbitrary continuous function on $G(d,j)$. Obviously, both valuations are even and they agree on all bodies $K\in{\mathcal K}^d_s$. This implies $\varphi_1=\varphi_2$ (e.g. since they have the same Klain function). Letting $g$ vary, we get the desired result. 

Using \eqref{gm-pgm} and Fubini's theorem, we also obtain a direct formula connecting $\gamma_j(K,\cdot)$ and $\rho_{(j)}(K,\cdot)$.

\begin{prop}\label{prop.4-1} For a generalized zonoid $K$ and a Borel set $A\subset G(d,j)$, we have
$$
\gamma_j(K,A) = \int_{G(d,j)} ({\bf C}_j \nu_j^A)(L)\rho_{(j)}(K,dL) ,
$$
where $\nu_j^A$ denotes the restriction of $\nu_j$ to $A$.
\end{prop}

\section{Proof of Theorem \ref{charact-1}}

If $\varphi\in{\bf Val}_j$ is even, it can be reconstructed from its Klain function ${\bf K}_j\varphi$ on the set of generalized zonoids. In fact, for a generalized zonoid $K$, Klain showed that
\begin{align*}
\varphi (K) = \frac{2^j}{j!} \int_{S^{d-1}}\cdots\int_{S^{d-1}} &({\bf K}_j\varphi)(\lin\{u_1,\dots,u_j\}) D_j(u_1,\dots ,u_j)\\
&\quad\times\rho(K,du_1)\dots\rho(K,du_j) 
\end{align*}
(see \cite[Theorem 6.4.12]{S}). Using the projection generating measure $\rho_{(j)}(K,\cdot)$, we can rewrite this as
\begin{align}\label{projgen}
\varphi (K) &=  \int_{G(d,j)}({\bf K}_j\varphi)(L)\, \rho_{(j)}(K,dL). 
\end{align}
Since projection generating measures cannot be extended to arbitrary bodies $K\in{\mathcal K}^d_s$ in a continuous way (see the discussion in \cite{HRW}), the question of a complete reconstruction is still open.  

We can give here a solution for valuations, for which the Klain function ${\bf K}_j\varphi$ is the cosine transform of a function $f\in C(G(d,j))$.

\begin{theorem}\label{complrec} Let $\varphi\in {\bf Val}_j^+$, $j\in\{1,\dots,d-2\}$ and $f\in C(G(d,j))$. Then, we have
\begin{equation}\label{valrepr}
\varphi (K) = \int_{G(d,j)} f(L) \gamma_j(K,dL)
\end{equation}
for all $K\in{\mathcal K}^d_s$, iff
\begin{equation}\label{croftonmeas2}
{\bf K}_j\varphi = {\bf C}_j f .
\end{equation}
\end{theorem}

\begin{proof} Assume ${\bf K}_j\varphi = {\bf C}_j f.$ 
Combining  \eqref{gm-pgm} with \eqref{projgen}, we get, for a generalized zonoid $K$,
\begin{align*}
\int_{G(d,j)} f(N)\gamma_j(K,dN) &= \int_{G(d,j)}\int_{G(d,j)} f(N)\langle M,N\rangle \rho_{(j)}(K,dM)\nu_j(dN)\\
&= \int_{G(d,j)}\left(\int_{G(d,j)} f(N)\langle M,N\rangle \nu_j(dN)\right)\rho_{(j)}(K,dM)\\
&= \int_{G(d,j)} ({\bf K}_j\varphi)(M)\, \rho_{(j)}(K,dM)\\
&= \varphi (K).
\end{align*}
This implies
$$
\varphi (K) = \int_{G(d,j)} f(N)\gamma_j(K,dN),
$$
for arbitrary $K\in{\mathcal K}^d_s$, by continuity, since generalized zonoids are dense in $ {\mathcal K}^d_s$.

In the other direction, let $L'\in G(d,j)$ and $K\in{\mathcal K}^d_s, K\subset L', V_j(K)>0$.
Then, our assumption and the definition of the Klain function yield
\begin{align*}
{\bf K}_j\varphi &(L')V_j(K)= \varphi(K)\\
&= \int_{F(d,j+1)} f(u^\perp\cap L) \psi_j(K,d(u,L))\\
&= \int_{G(d,j+1)} \int_{S^{d-1}} f(u^\perp\cap L) \Psi_j'(K|L,du)\nu_{j+1}(dL) .
\end{align*}
For $\nu_{j+1}$-almost all $L$, the projection $K|L$ is $j$-dimensional. Therefore, the area measure $\Psi_j'(K|L,\cdot)$ is concentrated on the two normals $u_0,-u_0$ of $K|L$ in $L$.  Since $V_j(K|L)=$ $V_j(K)\langle L,L'\rangle$, we get
\begin{align*}
{\bf K}_j\varphi &(L')V_j(K)\\
&= \int_{G(d,j+1)}  f(u_0^\perp\cap L) V_j(K|L)\nu_{j+1}(dL) \\
&= V_j(K) \int_{G(d,j+1)}  f(L'| L) \langle L',L'|L\rangle\nu_{j+1}(dL) \\
&= V_j(K) \int_{G(d,j)}  f(M) \langle L',M\rangle\nu_{j}(dM) \\
&= V_j(K) {\bf C}_j f (L'),
\end{align*}
which yields the assertion. 
\end{proof}

We remark, that condition \eqref{valrepr} for symmetric bodies extends to all $K\in{\mathcal K}^d$. This follows as in the proof of Theorem \ref{grassmann}.

We also remark that the proof of Theorem \ref{complrec} given here is completely independent of Theorem \ref{grassmann}. If we take Theorem \ref{grassmann} into account, then \eqref{valrepr} yields
$$
\varphi (K) = \int_{G(d,j)} f(L) V_j(K|L)\nu_j(dL) ,
$$
hence $\varphi$ has Crofton measure 
$$\mu = \int_{G(d,j)} {\bf 1}\{ L\in \cdot\} f(L)\nu_j(dL)$$ 
and in this case \eqref{croftonmeas2} reduces to \eqref{croftonmeas}. The equivalence of \eqref{valrepr} and \eqref{croftonmeas2} thus follows from the equivalence explained at the end of Section 3.

The range of the cosine transform on $G(d,j)$ was determined by Alesker and Bernstein \cite{AB} and it was shown that the Klain function of a smooth valuation $\varphi\in {\bf Val}_j^{\infty,+}$ lies in this range (in fact, the image of the Klain embedding of smooth valuations coincides with the image of the cosine transform on smooth functions). From this fact and Theorem \ref{complrec}, we get Theorem \ref{charact-1}, immediately. Alesker \cite{A1} stated the representation \eqref{alesker} in the proof of his Theorem 1.1.1.


Theorem \ref{complrec} indicates that a representation
$$
\varphi (K) = \int_{G(d,j)} f(L) \gamma_j(K,dL), \quad K\in {\mathcal K}^d_s,
$$
with a continuous function $f=f_\varphi$, is not possible for all $\varphi\in {\bf Val}_j^+$, if $j\in\{1,\dots,d-2\}$. This can be seen explicitly for the mixed volume
$$
\varphi (K) = V(K [j], M[d-j]),
$$
where $M\in {\mathcal K}^d_s$ is a fixed body. Assume
$$
\varphi (K) = \int_{G(d,j)} f_M(L) \psi_j(K,dL)
$$
holds for some function $f_M\in C(G(d,j))$ and all $K\in{\mathcal K}^d_s$. Then, Theorem \ref{grassmann} implies
$$
V(K [j], M[d-j]) = \int_{G(d,j)} f_M(L) V_j(K|L) \nu_j(dL).
$$
Choosing $K\subset L'$ for $L'\in G(d,j)$ yields
$$
V_{d-j}(K|L'^\perp) = \int_{G(d,j)} f_M(L) \langle L,L'\rangle \nu_j(dL).
$$
If we replace $L'^\perp$ by $\bar L$ and define $f_M^\perp(L) = f_M(L^\perp)$, we obtain
$$V_{d-j}(M|\bar L) = \left[{\bf C}_{d-j} \int_{(\cdot)} f_M^\perp (L)\nu_{d-j}(dL)\right] (\bar L) .$$
Since $M\in{\mathcal K}^d_s$ was arbitrary, this yields a contradiction, as there are bodies $M\in{\mathcal K}^d_s$, for which the projection function $V_{d-j}(M|\cdot)$ on $G(d,d-j)$ is not the cosine transform of a measure of the form $\int_{(\cdot)} f_M^\perp (L)\nu_{d-j}(dL)$. More precisely, for $2\le j\le d-2$,  $V_{d-j}(M|\cdot)$ need not be the cosine transform of a measure at all (see the discussion in \cite{HRW}), whereas in the case $j=1$, it is (proportional to) the cosine transform of the area measure $\Psi_{d-1}(M,\cdot)$, but the latter is not absolutely continuous, in general.

So far, the discussion was about the representation
$$
\varphi (K) = \int_{G(d,j)} f(L) \gamma_j(K,dL)
$$
for a valuation $\varphi\in {\bf Val}^+_j$. This is equivalent to
\begin{equation}\label{genrep}
\varphi (K) = \int_{F(d,j+1)} g(u,M) \psi_j(K,d(u,M)),
\end{equation}
where $g\in C(F(d,j+1))$ is of the form
$$
g(u,M)=f(u^\bot\cap M),\quad (u,M)\in F(d,j+1),
$$
for some $f\in C(G(d,j)$. Thus, a more general question would be to ask which even valuations $\varphi$ have a representation \eqref{genrep}, if general functions $g\in C(F(d,j+1))$ are allowed?

\section{Proof of Theorem \ref{projaver-1}}

We first describe the connection between the Grassmann measure $\gamma_j(K,\cdot)$ and the area measure $\Psi_j(K,\cdot)$ for centrally symmetric bodies. To explain this, we 
use the Radon transform $R_{k,1}$ which maps continuous functions $f$ on $G(d,k)$ to even functions on $S^{d-1}$ by
$$
R_{k,1}f (u) = \frac{1}{2}\int_{G(u,k)} f(L) \nu^u_k(dL),\quad u\in S^{d-1} .
$$
Here, $G(u,k)=G(-u,k)=\{ L\in G(d,k) : u\in L\}$, $\nu^u_k$ is the invariant probability measure on $G(u,k)$ and we assume $k\in\{ 2,\dots, d-1\}$.  By duality, $R_{k,1}$ can be extended to a mapping from finite measures on $G(d,k)$ to even measures on $S^{d-1}$. The transpose $R_{1,k}$ maps even functions and measures on $S^{d-1}$ to functions and measures on $G(d,k)$. We also denote the image of a measure $\rho$ on $G(d,k)$ under the orthogonality transform $G(d,k)\to G(d,d-k)$, $L\mapsto L^\perp$, by $\rho^\perp$.

\begin{theorem} For $K\in {\mathcal K}^d_s$ and $j=1,\dots,d-2$, we have
$$
\tilde\Psi_j(K,\cdot) = R_{d-j,1}\gamma_j^\perp(K,\cdot)
$$
where 
$$
\tilde\Psi_j(K,\cdot) = \int_{G(d,j+1)}\Psi_j(K|L,\cdot)\nu_{j+1}(dL)
$$
is the projection mean of the area measure $\Psi_j(K,\cdot)$ of $K$.
\end{theorem}

\begin{proof}
For $f\in C_o(S^{d-1}$), the Banach space of continuous functions on $S^{d-1}$ with centroid $0$, we consider the valuation $\varphi =\varphi_f$ given by
$$
\varphi (K) = \int_{S^{d-1}} f(u)\Psi_j(K,du),\quad K\in{\mathcal K}^d.
$$
For a body $K\in{\mathcal K}^d$ with dimension $j$, the $j$th order area measure is proportional to the spherical Lebesgue measure in the $(d-j)$-space orthogonal to $K$. Hence, the Klain function of $\varphi$ satisfies
$$
({\bf K}_j\varphi)(L) = c R_{1,d-j}f(L^\bot),\quad L\in G(d,j).
$$
Chosing $f=1$, we have ${\bf K}_j\varphi = R_{1,d-j}f =1$, hence $c=1$.
Using Theorem \ref{Grassmannmeas} and Fubini's theorem, we thus get 
\begin{align*}
\int_{S^{d-1}} f(u)&[R_{d-j,1}\gamma_j^\perp(K,\cdot)](du) \\
&=\int_{G(d,j)} R_{1,d-j}f(L^\bot) \gamma_j(K,dL)\\
&= \int_{G(d,j+1)} \int_{S^{d-1}}f(u)\Psi_j (K|L,du)\nu_{j+1}(dL)\\
&= \int_{S^{d-1}}f(u)\left[\int_{G(d,j+1)}\Psi_j (K|L,\cdot)\nu_{j+1}(dL) \right] (du)\\
&= \int_{S^{d-1}} f(u) \tilde\Psi_j(K,du) .
\end{align*}
Since this holds for all continuous functions $f$, we obtain the result.
\end{proof}

From the above proof and Theorem \ref{grassmann}, we obtain
\begin{align*}
\int_{S^{d-1}} f(u) \tilde\Psi_j(K,du) &= \int_{G(d,j)} R_{1,d-j}f(L^\bot) \gamma_j(K,dL)\\
&= \int_{S^{d-1}} f(u) [R_{d-j,1}V_j(K|\cdot^\bot)](u)\omega (du) ,
\end{align*}
for a continuous function $f$  and the invariant probability measure $\omega$ on $S^{d-1}$. 
Here, we have used Theorem 7.1.1 in \cite{SW}. Since
\begin{align*}
[R_{d-j,1}V_j(K|\cdot^\bot)](u)
&=\int_{G(u,d-j)} V_j(K|L^\perp)\nu_{d-j}^u(dL)\\
&=\int_{G(u^\perp,j)} V_j(K|L)\nu_{j}^{u^\perp}(dL)\\
&= c^{j,d-1-j}_{d-1,0}V_j(K|u^\perp) ,
\end{align*}
as follows from the global version of \eqref{projaver}, 
and letting $f$ vary, we get the case $k=j+1$ of the following result.

\begin{koro}\label{projaver-0} For $K\in{\mathcal K}^d_s$, $A\subset S^{d-1}$ and $1\le j <k \le d-1$, we have
\begin{equation}\label{aver}
\int_{G(d,k)}\Psi_j (K|L, A)\nu_{k}(dL)  = c^{k-1,d-k}_{d-1,0} \int_{A} V_j(K|u^\perp)\omega (du) .
\end{equation}
\end{koro}

In contrast to \eqref{projaver}, we integrate here the area measure of $K|L$ as a convex body in $\R^d$. The resulting measure is not locally determined by $K$ and therefore it is not a multiple of $\Psi_j(K,\cdot)$ (compare the characterization result for area measures in \cite{S75}).

\begin{proof} The remaining general case $k\ge j+2$ follows, if we apply the result for $j=k-1$ (which is already established) and use the Steiner formula for area measures and intrinsic volumes in $K|L, L\in G(d,k)$. 
\end{proof}

To finish the proof of Theorem  \ref{projaver-1}
we use the fact that, under our dimensional condition and for centrally symmetric $K$, the projection function $u\mapsto V_j(K|u^\perp)$ determines $K$ uniquely. In fact, $u\mapsto V_j(K|u^\perp)$ is (up to a constant) the cosine transform of the area measure $\Psi_j(K,\cdot)$ and the latter determines $K$.

\medskip
It is interesting to notice that in Theorem \ref{projaver-1} the symmetry of the bodies is essential. Namely, for $j=1, k=2$ and $d\ge 3$, we have
\begin{align*}
\tilde \Psi_1(K,\cdot) &=
\int_{G(d,2)}\Psi_1(K|L,\cdot)\nu_{2}(dL)\\
&= c\int_{G(d,2)}\square h^*(K|L,\cdot)\nu_{2}(dL)\\
&= c\square \left(\int_{G(d,2)} h^*(K|L,\cdot)\nu_{2}(dL)\right)\\
&= c\square h(P_2(K),\cdot)\\
&= \Psi_1(P_2(K),\cdot)\, ,
\end{align*}
where $c$ is a certain constant and $P_2(K)$ the second {\it projection mean body} of $K$ (for this notion, see \cite{Sch77}, \cite{G98} and the survey \cite{G00}), and where we have used the connection between the first area measure and the centred support function $h^\ast$ (via the block operator $\square$), see, e.g., \cite{GW92}.  As was shown by Goodey \cite{G98}, $P_2$ is not injective in dimension $d=14$, hence there are two different convex bodies $K,K'\subset \R^{14}$ with the same integral mean
$$
 \int_{G(d,2)}\Psi_1(K|L,\cdot)\nu_{2}(dL)=\int_{G(d,2)}\Psi_1(K'|L,\cdot)\nu_{2}(dL).
$$
Since in dimension 14, the fifth multiplier in the spherical harmonic expansion of $P_2$ vanishes (and the others are nonzero), see \cite{G98} again, at least one of the bodies $K,K'$ has to be asymmetric.

This case $j=1,k=2$ also shows that formula \eqref{aver} does not hold, in general, without the assumption of symmetry of $K$. Namely, for $d\not= 14$, the second projection body $P_2(K)$ does determine $K$ (up to translations). Hence $\tilde\Psi_1(K,\cdot)$ determines $K$, whereas the right side of \eqref{aver} obviously does not determine $K$, if $K$ is not symmetric. This indicates, that a uniqueness theorem for $K\mapsto \tilde\Psi_j(K,\cdot)$ might still hold for general $K$ (up to translations) and for certain values of $d$ and $j$. Even more, the cases $(d,j)$ where we have non-uniqueness might be rather sporadic, as has turned out in similar situations (see \cite{GW06a, GW06b}).

\medskip
We finish the section with a version of Corollary \ref{projaver-0} in the spirit of Proposition \ref{prop.4-1}.


\begin{prop}\label{prop.6-1} For  $K\in {\mathcal K}^d_s$, a Borel set $A\subset S^{d-1}$ and $1\le j<k\le d-1$, we have
$$
\int_{G(d,k)} S_j(K|L,A)\nu_k(dL) = \int_{S^{d-1}} [{\bf C}_{d-1} \Psi_j^A(K,\cdot)](u){\omega}(du) ,
$$
where $\Psi_j^A(K,\cdot)$ denotes the restriction of $\Psi_j(K,\cdot)$ to $A$.
\end{prop}

\section{Touching measures}

Let $\mu_q(K,\cdot)$ be the natural measure on the space $A(d,q,K)$ of affine $q$-flats in $\R^d$ touching the convex body $K$ (see \cite[Section 8.5]{SW}, for the definition of this touching measure and its properties). The image $\nu_q(K,\cdot)$ of $\mu_q(K,\cdot)$ under the mapping $\pi : E\mapsto L(E)$ is another measure on the Grassmannian $G(d,q)$ attributed to $K$. We call it the $q$th {\em direction measure} of $K$. The following result describes the connection between $\gamma_j(K,\cdot)$ and $\nu_{d-j-1}(K,\cdot)$. It involves the Radon transform
$$
R_{jk} : C(G(d,j))\to C(G(d,k)),
$$
defined by
\begin{equation}\label{Radon}
(R_{jk}f)(E) = \int_{G(E,j)} f(L)\nu_j^E(dL),\quad E\in G(d,k),
\end{equation}
where $0\le j,k\le d$, $G(E,j)$ is the Grassmannian of $j$-spaces containing $E$ (or contained in $E$, in case $j
<k$) and $\nu_j^E$ is the rotation invariant probability measure on $G(E,j)$. (The transform $R_{j,1}$ which we used earlier, corresponds to the case $k=1$ of \eqref{Radon},  if lines $E\in G(d,1)$ are identified with pairs $\{ u,-u\}, u\in S^{d-1}$.) By duality, $R_{jk}$ extends to a linear mapping on measures, 
$$
R_{jk} : C'(G(d,j))\to C'(G(d,k)),
$$
through
\begin{equation}\label{Radon2}
\int_{G(d,k)}f(L)(R_{jk}\rho)(dL) = \int_{G(d,j)} (R_{kj}f)(M)\rho(dM),\quad f\in C(G(d,k)).
\end{equation}

\begin{theorem} \label{direction}
The Grassmann measure $\gamma_j(K,\cdot)$ satisfies
$$
 \nu_{d-j-1}^\perp(K,\cdot)= 2 R_{j,j+1} \gamma_j(K,\cdot) .
$$
\end{theorem}

\begin{proof} We know from \cite[p. 358]{SW} that 
$$
\mu_{d-j-1}(K,A) = 2\int_{SO_d} \Phi_{j}(K|\vartheta L_{d-j-1}^\perp,T_{d-j-1}(A,\vartheta))\nu(d\vartheta) ,
$$
for all Borel sets $A\subset A(d,d-j-1,K)$, 
with
$$
T_{d-j-1}(A,\vartheta) = \{ x\in L_{d-j-1}^\perp : \vartheta L_{d-j-1}+x\in A\} .
$$
Here, $L_{d-j-1}$ is a fixed space in $G(d,d-j-1)$, $\nu$ is the Haar probability measure on the rotation group $SO_d$,  and $\Phi_j(K,\cdot)$ denotes the $j$th curvature measure of $K$ (a finite measure concentrated on the boundary of $K$). Obviously, this can be re-written as
\begin{equation}\label{touching}
\mu_{d-j-1}(K,A) = 2\int_{G(d,j+1)} \Phi_{j}(K|L,T(A,L))\nu_{j+1}(dL)
\end{equation}
where
$$
T(A,L) = \{ x\in L : L^\perp+x\in A\} .
$$
Hence
\begin{align*}
\nu_{d-j-1}(K,A) &= \mu_{d-j-1}(K,\pi^{-1}(A))\\
&=2\int_{G(d,j+1)} \Phi_{j}(K|L,T(\pi^{-1}(A),L))\nu_{j+1}(dL),\\
\end{align*}
where
\begin{align*}
T(\pi^{-1}(A),L)&= \{x\in L : x+L^\perp\in \pi^{-1}(A)\}\\
&= \{x\in\  {\rm relbd}\, (K|L) : L^\perp\in A\}.\\
\end{align*}
We obtain
\begin{equation}\label{Radontype}
\nu_{d-j-1}(K,A)=2\int_{G(d,j+1)} V_{j}(K|L){\bf 1}\{L^\perp\in A\}\nu_{j+1}(dL).
\end{equation}
In equivalent form, \eqref{Radontype} yields
\begin{align*}
\int_{G(d,j+1)} f(L) \nu_{d-j-1}^\perp (dL) &= 2\int_{G(d,j+1)} f(L) V_j(K|L)\nu_{j+1} (dL)\\
&= 2\int_{G(d,j+1)} \int_{G(L,j)}f(L) V_j(K|M)\nu_j^L(dM)\nu_{j+1} (dL) \\
&= 2\int_{G(d,j)} \int_{G(M,j+1)}f(L) V_j(K|M)\nu_{j+1}^M(dL)\nu_{j} (dM)\\
&= 2\int_{G(d,j)} (R_{j+1,j}f)(M)V_j(K|M)\nu_{j} (dM)\\
&= 2\int_{G(d,j)} (R_{j+1,j}f)(M)\gamma_{j} (K,dM).
\end{align*}
Here, we have used the Cauchy-Kubota formula for $V_j$ in $L$ (see \cite[Theorem 6.2.2]{SW}), the flag formula Theorem 7.1.1 in \cite{SW} (with the notation given there), and Theorem 1.1. 

The assertion follows now from \eqref{Radon2}.
\end{proof}

Letting $A$ vary in \eqref{Radontype} shows that $\nu_{d-j-1}(K,\cdot) = \nu_{d-j-1}(M,\cdot)$, for $K,M\in {\mathcal K}^d_s$, implies $V_j(K|L)=V_j(M|L)$, for all $L\in G(d,j+1)$. This proves Theorem \ref{unique-2}.

\medskip\noindent
{\bf Remark.} In the background of the considerations above is a connection between the touching measure $\mu_q(K,\cdot)$ and the flag support measure $\Theta^{(q)}_{d-q-1}(K,\cdot)$ discussed in \cite{HTW}. The latter is a measure on $\R^d\times S^{d-1}\times G(d,q)$ introduced as
\begin{align}\label{flagsupport}
\Theta^{(q)}_{d-q-1}(K,\cdot) = \int_{G(d,q)}\int {\bf 1}\{ (g(x,L,K),u,L)\in \cdot\} \Theta^{L^\perp}_{d-q-1}(K|L^\perp,d(x,u))\nu_q(dL)
\end{align}
(see \cite[Theorem 4]{HTW}). Here, $\Theta^{L^\perp}_{d-q-1}(K|L^\perp,\cdot)$ is the (highest-order) support measure of $K|L^\perp$ calculated in $L^\perp$ as the ambient space and $g(x,L,K)$ is the point in the boundary of $K$ such that the affine flat $L+x$ touches $K$ in $g(x,L,K)$. As is shown in \cite{HTW}, this contact point is unique (and measurable) for given $K$ and $\nu_q$-almost all $L\in G(d,q)$. 

Since $L+x = L+g(x,L,K)$ for boundary points $x$ of $K|L^\perp$, the image measure $\tilde\mu_q(K,\cdot)$ of $\Theta^{(q)}_{d-q-1}(K,\cdot)$ under the mapping $(x,u,L)\mapsto L+x$ is a measure on $A(d,q,K)$, satisfying
\begin{align*}
\tilde\mu_q(K,\cdot)&= \int_{G(d,q)}\int {\bf 1}\{L+x\in \cdot\} \Theta^{L^\perp}_{d-q-1}(K|L^\perp,d(x,u))\nu_q(dL)\\
&= \int_{G(d,q)}\int {\bf 1}\{L+x\in \cdot\} \Phi_{d-q-1}(K|L^\perp,dx)\nu_q(dL)\\
&= \int_{G(d,d-q)} \Phi_{d-q-1}(K|L^\perp,T(\cdot|L))\nu_{d-q}(dL)\\
&= \frac{1}{2}\mu_q(K,\cdot) ,
\end{align*}
as follows from \eqref{touching}.

The direction measure $\nu_q(K,\cdot)$ of $K$ is thus (up to the factor $\frac{1}{2}$) the image of $\tilde\mu_q(K,\cdot)$ under the mapping $E\mapsto L(E), E\in A(d,q)$. Therefore, $\nu_q(K,\cdot)$ is the image of $\Theta^{(q)}_{d-q-1}(K,\cdot)$ under $(x,u,L)\mapsto L$. Splitting the latter map into $(x,u,L)\mapsto (u,L)$ and $(u,L)\mapsto L$ and using the fact that the image of $\Theta^{(q)}_{d-q-1}(K,\cdot)$ under $(x,u,L)\mapsto (u,L)$ is $S^{(q)}_{d-q-1}(K,\cdot)$ (which is connected to $\psi_{d-q-1}(K,\cdot)$ via \eqref{connect}), we conclude that $\nu_q(K,\cdot)$ is (up to a constant) the image of the flag measure $\psi_{d-q-1}(K,\cdot)$ under the mapping $(u,L)\mapsto L^\perp$. As Theorem 7.1 shows, the image of $\psi_{d-q-1}(K,\cdot)$ under $(u,L)\mapsto u^\perp\cap L$ corresponds to taking the Radon transform $R_{d-q,d-q-1}$. 

\section*{Acknowledgements}

This research has been supported by the DFG project WE 1613/2-2.

\noindent
I thank Andreas Bernig and Franz Schuster for useful remarks on a previous version of the paper.


\begin{thebibliography}{99}

\bibitem{A1} S. Alesker, Hard Lefschetz theorem for valuations, complex integral geometry, and unitarily invariant valuations. {\em J. Diff. Geom.} {\bf 63}, 63--95 (2003).




\bibitem{AB} S. Alesker, J. Bernstein, Range characterization of the cosine transform on higher Grassmannians. {\em Adv. Math.} {\bf 184}, 367--379 (2004).


\bibitem{Bernig} A. Bernig, Valuations with Crofton formula and Finsler geometry. {\em Adv. Math.} {\bf 210}, 733--753 (2007).






\bibitem{G98} P. Goodey, Minkowski sums of projections of convex bodies. {\em Mathematika} {\bf 45}, 253--268 (1998).

\bibitem{GHHRW} P. Goodey, W. Hinderer, D. Hug, J. Rataj, W. Weil, A flag representation of projection functions. {\em Adv. Geom.} (to appear) (2017).



\bibitem{GHW2} P. Goodey, D. Hug, W. Weil, Section and projection formulas for homogeneous valuations. {\em In preparation} (2017+).

\bibitem{G00} P. Goodey, W. Jiang, Minkowski sums of three dimensional projections of convex bodies. {\em Rend. Circ. Mat. Palermo, Ser. II, Suppl.} {\bf 65}, 105--119 (2000).

 
\bibitem{GW92} P. Goodey, W. Weil, Centrally symmetric convex bodies and the spherical Radon transform. {\em J. Diff. Geom.} {\bf 35}, 675--688 (1992).

\bibitem{GW93} P. Goodey, W. Weil, Zonoids and generalisations. In {\em Handbook of Convex Geometry} (P.M. Gruber, J.M. Wills, eds.), vol. B, pp. 1297--1326, North-Holland, Amsterdam, 1993.

\bibitem{GW06a} P. Goodey, W. Weil, Average section functions for star-shaped sets. {\it Adv. Appl. Math.} {\bf 36}, 70--84 (2006).

\bibitem{GW06b} P. Goodey, W. Weil, Directed projection functions of convex bodies. {\it Monatsh. Math.} {\bf 149}, 43--64, 65 (Erratum) (2006).





\bibitem{Hind} W. Hinderer, {\em Integral Representations of Projection Functions}. PhD Thesis, University of Karlsruhe, Karlsruhe 2002. 

\bibitem{HHW} W. Hinderer, D. Hug, W. Weil, Extensions of translation invariant valuations on polytopes. {\it Mathematika} {\bf 61}, 236--258 (2015).


\bibitem{HRW} D. Hug, J. Rataj, W. Weil, A product integral representation of mixed volumes of two convex bodies. {\em Adv. Geom.} {\bf 13}, 633--662 (2013).

\bibitem{HRW2} D. Hug, J. Rataj, W. Weil, Flag representations of mixed volumes and mixed functionals of convex bodies. {\em Submitted}, arxiv 1705.04816 (2017).

\bibitem{HTW} D. Hug, I. T\"urk, W. Weil, Flag measures for convex bodies. In: {\it Asymptotic Geometric Analysis}, ed. by M. Ludwig et al., Fields Institute Communications, Vol. {\bf 68}, Springer, 2013, 145--187.




\bibitem{McM77} P. McMullen, Valuations and Euler-type relations on certain classes of convex polytopes. {\em Proc. London Math. Soc. (3)} {\bf 35}, 113--135 (1977).

\bibitem{McM80} P. McMullen, Continuous translation-invariant valuations on the space of compact convex sets. {\em Arch. Math.} {\bf 34}, 377--384 (1980).



\bibitem{McM93}P. McMullen,  Valuations and dissections. In: Gruber, P.M., Wills, J.M. (eds), {\em Handbook of Convex Geometry, vol. B}, pp. 933--988, North-Holland, Amsterdam 1993.

\bibitem{S75} R. Schneider, Kinematische Ber\"uhrma{\ss}e f\"ur konvexe K\"orper. {\em Abh. Math. Sem. Univ. Hamburg} {\bf 44}, 12--23 (1975).

\bibitem{Sch75}  R. Schneider, Kinematische Ber\"uhrma{\ss}e f\"ur konvexe K\"orper und Integralrelationen f\"ur Oberfl\"achenma{\ss}e. {\em Math. Ann.} {\bf 218}, 253--267 (1975). 

\bibitem{Sch77}  R. Schneider, Rekonstruktion eines konvexen K\"orpers aus seinen Projektionen. {\em Math. Nachr.} {\bf 79}, 325--329 (1977). 


\bibitem{S} R. Schneider, {\em Convex Bodies: The Brunn-Minkowski Theory} (Second Expanded Edition). Cambridge University Press, Cambridge 2013.

\bibitem{SW} R. Schneider, W. Weil, {\em Stochastic and Integral Geometry}. Springer, Heidel\-berg-New York 2008.

\bibitem{Schuster} F.E. Schuster, Crofton measures and Minkowski valuations. {\em Duke Math. J.} {\bf 154}, 1--30 (2010).






\end{thebibliography}
\end{document}